\documentclass[10pt,a4paper,reqno]{amsart}
\usepackage[utf8]{inputenc}
\usepackage{amssymb}
\usepackage{amsthm}
\usepackage{amsmath}
\usepackage{mathtools}
\usepackage{color}
\usepackage{hyperref}
\usepackage{amsrefs}
\usepackage{cleveref}
\usepackage{tikz}
\usepackage{cases}
\usepackage{enumerate}

\numberwithin{equation}{section}

\theoremstyle{definition}
 \newtheorem{theorem}{Theorem}[section]
 \crefname{theorem}{Theorem}{Theorems}
 \newtheorem{proposition}[theorem]{Proposition}
 \crefname{proposition}{Proposition}{Propositions}
 \newtheorem{lemma}[theorem]{Lemma}
 \crefname{lemma}{Lemma}{Lemmas}
 
 \crefname{corollary}{Corollary}{Corollaries}
 
 \crefname{conjecture}{Conjecture}{Conjectures}
 
 \crefname{question}{Question}{Questions}
 
 \crefname{problem}{Problem}{Problems}
 \newtheorem{remark}[theorem]{Remark}
 \crefname{remark}{Remark}{Remarks}

\theoremstyle{definition} 
 
 \crefname{definition}{Definition}{Definitions}
 
 \crefname{example}{Example}{Examples}
 
 \crefname{caution}{Caution}{Cautions}
 
 \crefname{equation}{formula}{formulas}

\newcommand{\abs}[1]{\left\lvert#1\right\rvert}

\newcommand{\setmid}{\mathrel{}\middle|\mathrel{}} 

\newcommand{\F}{\mathbb{F}}

\newcommand{\Q}{\mathbb{Q}}
\newcommand{\R}{\mathbb{R}}
\newcommand{\Z}{\mathbb{Z}}

\DeclareMathOperator{\Disc}{Disc} %discriminant
\DeclareMathOperator{\Ker}{Ker} %kernel
\DeclareMathOperator{\ord}{ord} %order
\DeclareMathOperator{\rank}{rank} %rank
\newcommand{\Sel}[2]{\mathrm{Sel}_{#1}(#2)} %添字, 引数込みのSelmer群
\usepackage[OT2,T1]{fontenc}
\DeclareSymbolFont{cyrletters}{OT2}{wncyr}{m}{n}
\DeclareMathSymbol{\Sha}{\mathalpha}{cyrletters}{"58}
\title{On the proportions of soluble forms in some families of locally soluble binary quartic forms}

\author{Yasuhiro Ishitsuka}
\address{Institute of Mathematics for Industry, Kyushu University, Fukuoka, 819-0395, Japan}
\email{yishi1093@gmail.com}
\author{Yoshinori Kanamura}
\address{Department of Mathematics \\ Faculty of Science and Technology, Keio University, 3-14-1, Hiyoshi, Kohoku, Yokohama, Kanagawa, Japan}
\email{kana1118yoshi@keio.jp}


\begin{document}

\maketitle

\begin{abstract}
    An integral binary quartic form is said to be locally soluble (resp.\ soluble) if the corresponding genus one curve has a rational point over $\mathbb{Q}_v$ for every place $v$ of $\mathbb{Q}$ (resp.\ over $\mathbb{Q}$). We consider the proportion of soluble integral binary quartic forms in locally soluble forms.
    Bhargava showed the proportion is positive when one considers all binary quartics, and Bhargava--Ho proved the proportion is zero for a subfamily.
    
    In this paper, we estimate the proportions for some other subfamilies. 
    It relies on results for elliptic curves $y^2=x^3-n^2x$ by Heath-Brown, Xiong--Zaharescu and Smith.
\end{abstract}

%----------------------
\section{Introduction}
%----------------------
Let $f(x,y)$ be a binary quartic form over $\Q$.
We call the curve $C_f\colon z^2=f(x,y)$ \emph{locally soluble} if $C_f(\Q_v)\neq\emptyset$ for all places $v$ of $\Q$.
We also call the curve $C_f$ \emph{soluble} if $C_f(\Q)\neq\emptyset$.
In what follows, 
we also call $f(x,y)$ locally soluble (resp.\ soluble) when $C_f$ is locally soluble (resp.\ soluble).
For an integral binary quartic form $f$, 
the naive height means the maximum of absolute values of the coefficients of $f$.
In 2014, 
Bhargava showed the following interesting result.

\begin{theorem}[cf.\ {\cite[Theorem 2]{Bhargava2014}}]\label{Bhargava}
When locally soluble integral binary quartics $f(x, y)$ are ordered by the naive height, the proportion of soluble forms is positive.
\end{theorem}
In addition to this theorem,
Bhargava conjectured the exact value of the proportion; for details, see \cite[Section 1.3]{Bhargava2014}.
Note that the proportion of locally soluble forms in integral binary quartic forms is already determined by Bhargava--Cremona--Fisher \cite[Theorem 3]{BCF2021}.
Hence, Bhargava also conjectured the exact value of the proportion of soluble forms in integral binary quartics ordered by the naive height.

Recently, 
Bhargava--Ho \cite{Bhargava-Ho2022} showed an analogue theorem on a subfamily of binary quartics.
Before we state their result,
we introduce the \emph{Bhargava--Ho height}.
For an integral binary quartic $f = ax^4 + Bx^2y^2 + cy^4$, we define the Bhargava--Ho height $H_{\text{BH}}$ as
\begin{align*}
H_{\text{BH}}(f) = \max\{B^2, \abs{ac}\}.
\end{align*}
Then, Bhargava--Ho's result is the following.
\begin{theorem}[{\cite[Theorem 1.6]{Bhargava-Ho2022}}]
When locally soluble integral binary quartic forms $ax^4 + Bx^2y^2 + cy^4$ are ordered by the Bhargava--Ho height,
the proportion of soluble forms is $0\%$.
\end{theorem}

Comparing Bhargava's and Bhargava--Ho's results,
we find that the proportions are totally different in the two cases.
In this paper,
we examine the proportion of soluble forms in some other subfamilies.
Especially, we find some subfamilies whose proportions of soluble forms are totally different.

For an integer $B$ and $M$, 
we write $W^B_M(\Z)$ for the set of binary quartic forms $f(x,y)=ax^4+Bx^2y^2+cy^4$ with $a,c \in \Z$ and $ac=M$. 
We also define
\begin{align*}
\mathcal{F}^{B}_{M} = \{ax^4+Bx^2y^2+cy^4\in W^B_{M}(\Z) \mid \text{$0\leq \ord_p a\leq 1$ for all places $v$ of $\Q$}\}.
\end{align*}
In what follows,
we abbreviate squarefree as ``sqf.''.

\begin{theorem}[{see \cref{main1refined}}]\label{main1}
When locally soluble forms $f\in \bigcup_{\substack{n\geq 1\\ \text{$n$: sqf.}}}\mathcal{F}^{0}_{4n^2}$ are ordered by the Bhargava--Ho height,
the proportion of soluble forms is $100\%$.
\end{theorem}

By slightly changing the condition on the coefficients,
we obtain the exact opposite result.

\begin{theorem}[{see \cref{phihat1}}]\label{main2}
When locally soluble forms $f \in \bigcup_{\substack{n\geq 1\\ \text{$n$: sqf.}}}\mathcal{F}^{0}_{-n^2}$ are ordered by the Bhargava--Ho height, the proportion of the soluble forms is $0\%$.
\end{theorem}
Note that we also obtain similar theorems for four slightly different subfamilies (see \cref{main1variant} and \cref{phihat2and3}). 

Bhargava's and Bhargava--Ho's results are deeply related to the average size of Selmer groups in families of elliptic curves.
For example, to prove \cref{Bhargava}, we first interpret the locally soluble binary quartics as an element of 2-Selmer group of an elliptic curve. Then \cref{Bhargava} is obtained by using the average size of $2$-Selmer groups in all elliptic curves.
Similarly, to prove our results, \cref{main1,main2}, we interpret locally soluble binary quartic forms in $\mathcal{F}_M^{B}$ as
elements of $2$-isogeny or its dual isogeny Selmer groups in the family of elliptic curves $E_n \colon y^2=x^3-n^2x$ where $n$ runs over all squarefree positive integers.
Then the results follow from the average sizes of those Selmer groups.
The average sizes of these Selmer groups in the family have been studied a lot,
for example \cite{Aoki,Feng1996,Feng-Xiong2004,Feng-Xue,Heath-Brown1993,Heath-Brown1994,Lozano-Robledo,Vrecica,Xiong-Zaharescu2009,Yu2005,Yu2006}.
In particular,
our results are deeply relied on the results on \cite{Heath-Brown1993,Xiong-Zaharescu2009}.

Key facts to prove \cref{main1,main2} are
the difference in the order between the number of locally soluble quartics and the number of soluble quartics.
Especially, in the case of \cref{main2}, 
there exist much more locally soluble quartics than soluble quartics.
Hence, we next consider whether there are subsets of locally soluble forms comparable to soluble ones in the sense of order.
To answer this problem,
we introduce the condition which is called  \emph{strictly locally soluble}.

Let $n$ be a squarefree integer.
A curve $C_f\colon z^2=f(x,y)$ where $f\in \mathcal{F}^{0}_{-n^2}$ is called strictly locally soluble if the curve ``comes from'' $2$-Selmer groups of $E_n$.
We will explain the meaning of ``comes from'' explicitly in Section \ref{proofofmain3}.
In a similar manner to a locally soluble quartic,
we call a binary quartic form $f(x,y)$ strictly locally soluble when $C_f$ is strictly locally soluble.
Then, 
we answer the above problem of comparable subsets of locally soluble quartics, which is our last theorem.

\begin{theorem}[{see \cref{keytheorem}}]\label{main3} 
The proportion of soluble forms is greater than $42\%$ when strictly locally soluble forms $f\in\bigcup_{\substack{n\geq 1\\ \text{$n$: sqf.}}}\mathcal{F}^{0}_{-n^2}$ are ordered by the Bhargava--Ho height.
\end{theorem}

To prove this theorem,
we combine Heath-Brown's result \cite{Heath-Brown1993} and Browning's discussion \cite[Proof of Theorem 1.3]{Browning2017} (cf.\ \cite[Proof of Theorems $1$ and $2$]{Bhargava2014}).
Actually, we can show easily that the proportion appeared in \cref{main3} is strictly positive.
This is because we can immediately find soluble elements, for example $x^4-n^2y^4, -x^4+n^2y^4, nx^4-ny^4, -nx^4+ny^4\in\mathcal{F}^{0}_{-n^2}$.
These elements correspond to $2$-torsions of $E_n(\Q)$.
Our theorem states that there are still many soluble elements even if we remove the soluble elements corresponding to $2$-torsion points in $E_n(\Q)$.
To prove this theorem,
Smith's result \cite{Smith} on the proportion of $E_n$ with positive ranks plays an important role.

The plan of this paper is as follows.
In Section \ref{setting},
we prepare some notations or properties which we use in this paper.
In Section \ref{integral},
we introduce and prove the integrality lemma.
This lemma plays an important role when we consider the relationship between integral binary quartic forms and Selmer groups of $E_n$ for squarefree integers $n$.
In Section \ref{proofofmain1} (resp.\ Section \ref{proofofmain2}, and Section \ref{proofofmain3}), we state the explicit forms of \cref{main1} (resp.\ \cref{main2} and \cref{main3}) and give their proof.

%----------------------
\section{Setting}\label{setting}
%----------------------

For an integer $B$ and $M$, 
we set
\begin{align*}
W^B_M(\Q)&=\{f(x,y)=ax^4+Bx^2y^2+cy^4 \mid a, c \in \Q, ac=M\},\\
W^B_M(\Z)&=\{f(x,y)=ax^4+Bx^2y^2+cy^4 \mid a, c \in \Z, ac=M\}.
\end{align*}
The discriminants and Bhargava--Ho height of $f \in W^B_M(\Q)$ is determined by $B, M$ since
\begin{align}
    %I(f) &= B^2 + 12M, \\
    %J(f) &= -2B(B^2 - 36M), \\
    \Disc(f) &= 16M(B^2 - 4M)^2, \label{DiscDef}\\
    H_{\mathrm{BH}}(f) &= \max\{ B^2, |M| \}. \label{BHht}
\end{align}
In this paper, we mainly consider the \emph{nondegenerate} quartics, i.e., the quartics with $\Disc(f) \neq 0$.
By the above description, the nondegeneracy of $f \in W^B_M(\Q)$ only depends on $B$ and $M$.

A nondegenerate quartic $f(x,y)\in W^B_M(\Q)$ defines
 a genus one curve $C_f\colon z^2=f(x,y)$ over $\Q$.
We write the subsets of $W^{B}_M(\Q)$ of locally soluble (resp.\ soluble) binary quartic forms as $W^B_M(\Q)^{\text{ls}}$ (resp.\ $W^B_M(\Q)^{\text{sol}}$).  In concrete terms, 
\begin{align*}
    W^B_M(\Q)^{\text{ls}} &= \{f\in W^B_M(\Q)\mid \text{$C_{f}(\Q_v)\neq\emptyset$ for all place $v$ of $\Q$} \},\\
    W^B_M(\Q)^{\text{sol}} &= \{f\in W^B_M(\Q)\mid \text{$C_f(\Q)\neq \emptyset$}\}. 
\end{align*}
Similarly we define for $W^B_M(\Z)^{\text{ls}}$ and $W^B_M(\Z)^{\text{sol}}$.

We define an equivalence relation
\begin{align*}
    &ax^4+Bx^2y^2+cy^4 \sim a'x^4+Bx^2y^2+c'y^4
\end{align*}
on $W^B_M(\Q)$ if $a'=s^2a$ and $c'=s^{-2}c$ for some $s \in \Q^*$.
This equivalence preserves solubility and local solubility.
We write $[f]$ as an equivalent class of $f\in W^B_M(\Q)$.
We set
\begin{align*}
    \mathcal{F}^{B}_M(\Q) =
    \{ax^4+Bx^2y^2+cy^4\in W^B_M(\Q) \mid \text{$0\leq \ord_p a\leq 1$ for all places $v$ in $\Q$}\}.
\end{align*}
Then,
we have $\mathcal{F}^{B}_M = \mathcal{F}^{B}_M(\Q) \cap W^B_M(\Z)$. The set
$\mathcal{F}^B_M(\Q)$ (resp.\ $\mathcal{F}^B_M$) is a complete system of representatives of the equivalence classes in $W^B_M(\Q)$ (resp.\ $W^B_M(\Z)$).
We also set
\begin{align*}
    \mathcal{F}^{B, \text{sol}}_M &= \mathcal{F}^B_M \cap W^{B}_{M}(\Z)^{\text{sol}},\\
    \mathcal{F}^{B, \text{ls}}_M &= \mathcal{F}^B_M \cap W^{B}_{M}(\Z)^{\text{ls}}.
\end{align*}

As we mentioned in the introduction, 
we consider the elliptic curves $E_n$ and over $\Q$ defined as
\begin{align*}
    E_n \colon y^2 = x^3 - n^2x
\end{align*}
for an integer $n > 0$.
We write the point of infinity of $E_n$ as $\infty$.
The curves $E_n$ have three types of isogenies:
\begin{align*}
    \varphi_1 &\colon E_n \to E_{1,n} \colon
    y^2=x^3+4n^2x, & (x,y) &\mapsto \left( \frac{y^2}{x^2}, -\frac{y(n^2+x^2)}{x^2} \right), \\
    \varphi_2 &\colon E_n \to E_{2,n}\colon y^2=x(x^2-6nx+n^2), & (x,y) &\mapsto \left( \frac{y^2}{(x+n)^2}, \frac{y(2n^2-(x+n)^2)}{(x+n)^2} \right), \\
    \varphi_3 &\colon E_n \to E_{3,n}\colon y^2=x(x^2+6nx+n^2), & (x,y) &\mapsto \left( \frac{y^2}{(x-n)^2}, \frac{y(2n^2-(x-n)^2)}{(x-n)^2} \right).
\end{align*}
For $i = 1,2,3$, we write the dual isogenies of $\varphi_i$ as $\widehat{\varphi_i}$.
We write the Selmer group and weak Mordell--Weil group corresponding to the isogeny $\varphi_i$ 
as $\Sel{\varphi_i}{E_n}$ and $E_{i,n}(\Q)/\varphi_i(E_n(\Q))$ and similar for $\widehat{\varphi_i}$.
The relation between these isogenies and the set $W^{B}_{M}(\Q)$ is summarized in the following lemma.

\begin{lemma}[{cf. \cite[Proposition X.4.9]{SilvermanAEC}}]\label{correspond}
For any integer $n \neq 0$, we obtain
\begin{align*}
    &\Q^*/\Q^{*2}\cong H^1(\Q, E_n[\varphi_1])\cong W^{0}_{4n^2}(\Q)/{\sim}, & d &\mapsto \left[dx^4+\frac{4n^2}{d}y^4 \right],\\
    &\Q^*/\Q^{*2}\cong H^1(\Q, E_n[\varphi_2])\cong W^{-6n}_{n^2}(\Q)/{\sim}, & d &\mapsto \left[ dx^4 - 6nx^2y^2+(n^2/d)y^4\right],\\ 
    &\Q^*/\Q^{*2}\cong H^1(\Q, E_n[\varphi_3])\cong W^{6n}_{n^2}(\Q)/{\sim}, & d &\mapsto \left[ dx^4 + 6nx^2y^2+(n^2/d)y^4\right],\\
    &\Q^*/\Q^{*2}\cong H^1(\Q, E_{1,n}[\widehat{\varphi_1}])\cong W^{0}_{-n^2}(\Q)/{\sim}, & d &\mapsto \left[dx^4-\frac{n^2}{d}y^4 \right],\\
    &\Q^*/\Q^{*2}\cong H^1(\Q, E_{2,n}[\widehat{\varphi_2}])\cong W^{3n}_{2n^2}(\Q)/{\sim}, & d &\mapsto \left[dx^4 + 3nx^2y^2+2(n^2/d)y^4\right],\\ 
    &\Q^*/\Q^{*2}\cong H^1(\Q, E_{3,n}[\widehat{\varphi_3}])\cong W^{-3n}_{2n^2}(\Q)/{\sim}, & d &\mapsto \left[ dx^4- 3nx^2y^2+2(n^2/d)y^4\right]. 
\end{align*}

In each case, the set of equivalence classes of locally soluble quartics (resp.\ soluble quartics) corresponds to the Selmer groups  (resp.\ weak Mordell--Weil groups).
\end{lemma}

For the quartics in one of six sets appearing in \cref{correspond}, 
the discriminant and Bhargava--Ho height of the quartics only depends on $n$ by \eqref{DiscDef} and \eqref{BHht}. 
In the followings, we mainly consider those quartics with $n \neq 0$.
In particular, we do not consider degenerate quartics.

%----------------------
\section{integrality lemma}\label{integral}
%----------------------

In this section, 
we give an integrality lemma analogous to \cite[Lemma 3.2]{Browning2017}. It states that locally soluble quartics in $W^{B}_{M}(\Q)$ is equivalent to integral ones.
It is presumably well-known to experts; for example,
see \cite[Proposition X.4.9]{SilvermanAEC} for treatments except $p=2$.
We include them for reader's convenience.

\begin{lemma}\label{integralSummarized}
Fix integers $B, M \in \Z$ with $M(B^2 - 4M) \neq 0$. 
Any equivalence class of $W^{B}_{M}(\Q)^{\text{ls}}$ contains a unique element of $\mathcal{F}^{B}_{M}$. 
\end{lemma}

\begin{proof}
    The uniqueness follows from the fact that $\mathcal{F}^{B}_{M}(\Q)$ contains a unique element in each equivalence class of $W^{B}_M(\Q)^{\text{ls}}$.
    We have to show that for each equivalence class $[f]$ of $W^{B}_{M}(\Q)$, the unique element in $[f] \cap \mathcal{F}^{B}_{M}(\Q)$ has integral coefficients.

    Take a locally soluble quartic $f(x,y) = ax^4 + Bx^2y^2 + cy^4 \in W^{B}_{M}(\Q)^{\mathrm{ls}}$.
    Assume that $f \in \mathcal{F}^{B}_{M}(\Q)$, or equivalently, that $a \in \Z$ and $0 \le \ord_p a \le 1$ for any prime $p$. 
    Since $a, B \in \Z$, we only have to show that $c \in \Z$. It is enough to prove  $\ord_p c \ge 0$ for any prime $p$.
    
    For a prime $p$ dividing $M$, we have $\ord_p c \ge 0$ since $ac = M$.
    Thus, we may assume that $p \nmid M$.
    Moreover, since $ac = M \in \Z$ and $\ord_p a \le 1$, we only have to consider the case when $(\ord_p a, \ord_p c) = (1, -1)$.
    
    Since $f$ is locally soluble, we may take a nontrivial solution $(Z, X, Y)$ of $z^2 = f(x,y)$ in $\Q_p$.
    By scaling, we may assume $X, Y, Z \in \Z_p$.
    Then, the followings hold:
\begin{itemize}
    \item The two values $\ord_p(aX^4), \ord_p(cY^4)$ are odd and distinct because 
    \begin{align*}
        \text{$\ord_p(aX^4) \equiv 1\bmod{4}$ and $\ord_p(cY^4) \equiv -1 \bmod{4}$}.
    \end{align*}
    \item We would like to compare $\ord_p(nX^2Y^2)$ with $\ord_p(aX^4)$ and $\ord_p(cY^4)$. Since $p \nmid M$, we compute
    \begin{align*}
    \ord_p (aX^4)+\ord_p (cY^4)&=\ord_p (acX^4Y^4)\\
    &=\ord_p (M)+\ord_p (X^4Y^4)\\
    &=\ord_p(X^4Y^4)\\
    &=2\ord_p(X^2Y^2).
    \end{align*}
    Thus, we obtain $(\ord_p (aX^4)+\ord_p(cY^4))/2=\ord_p (X^2Y^2)$.
\end{itemize}

    Combining the above two observations, we have
    \begin{align*}
        &\ord_p(aX^4)>\ord_p(X^2Y^2)>\ord_p(cY^4), \\
        \text{ or } &\ord_p(aX^4)<\ord_p(X^2Y^2)<\ord_p(cY^4).
    \end{align*}
    It is enough to consider the former case.
    Since $B \in \Z$, we have two inequalities;
    \begin{align*}
        \text{$\ord_p(aX^4)>\ord_p(cY^4)$ and $\ord_p(BX^2Y^2)>\ord_p(cY^4)$}.
    \end{align*}
    These conclude
    \begin{align*}
        \ord_p(aX^4 + BX^2Y^2 + cY^4) = \ord_p(cY^4),
    \end{align*}
    and the value is odd. 
    However, it contradicts to $Z^2 = aX^4 + BX^2Y^2 + cY^4$ since $\ord_p(Z^2)$ is even.
\end{proof}
Note that we can extend this lemma when $M\neq 0$, but we only consider nondegenerate quartics here.

Applying \cref{integralSummarized} to six cases in \cref{correspond}, we see that each element in Selmer groups of the isogenies $\varphi_i, \widehat{\varphi_i}$ has a unique representative of binary quartics in $\mathcal{F}^{B}_M$ for appropriate $B, M$.
Thus to count integral locally soluble forms in those $\mathcal{F}^{B}_{M}$ exactly corresponds to
count elements of corresponding Selmer groups.

%----------------------
\section{Proof of \cref{main1}}\label{proofofmain1}
%----------------------
Now we prove \cref{main1}.
For $h \in \{\pm 1, \pm 2, \pm 3\}$, define
\begin{align*}
    S(X, h)= \{D\in\Z \mid D\equiv h \bmod{8}, \text{$1 \leq D\leq X$, $D$ : squarefree}\}.    
\end{align*}
The number of elements is estimated as
\begin{align*}
    \# S(X, h) &= \frac{1}{\pi^2} X + o(X)
\end{align*}
as $X\to\infty$.
For more detailed estimation, see the proof of  \cite[Lemma 14]{Xiong-Zaharescu2009}.

\begin{lemma}[{\cite[Theorem 6]{Xiong-Zaharescu2009}}]\label{Selmergrp}
For $i\in\{1,2,3\}$,
one has
\begin{align*}
    \sum_{n\in S(X,h)}\#\Sel{{\varphi_i}}{E_{n}} &= \#S(X,h) + o(X)
    \end{align*}
    as $X\to \infty$.
    In particular, the contribution of nontrivial elements in $\Sel{{\varphi_i}}{E_{n}}$ is estimated as $o(X)$.
\end{lemma}

\begin{proof}
For $h=\pm 1, \pm 3$, see the proof of \cite[Theorem 6]{Xiong-Zaharescu2009}.

In the following, we consider the case $h=\pm 2$ and $i=1$. 
Suppose $p \neq 2$.
The conditions on $n, d$ for the existence of $\Q_p$-rational points on $C_{1,d}$ is given by \cite[Lemma 3.1 (1)-(3)]{Feng-Xiong2004}.
As the case $h=\pm 1, \pm 3$, we obtain
\begin{align*}
    \sum_{n \in S(X, h)} \#\Sel{\varphi_1}{E_{n}} 
    &\leq \sum_{n=dd'}\prod_{\substack{p\mid d\\ p\neq 2}} \frac{1}{4}\left(\left(\frac{-1}{p}\right)+1\right)\left(\left(\frac{d'}{p}\right)+1\right)\prod_{\substack{p\mid d'\\ p\neq 2}}\frac{1}{2}\left(\left(\frac{d}{p}\right)+1\right)\\
    &=\#S(X, h)+o(X).
\end{align*}
as $X\to \infty$.
Here, $\left(\frac{\cdot}{\cdot}\right)$ is the Legendre symbol.
Moreover, 
we obtain 
\begin{align*}
    \#S(X, h) \leq \sum_{n\in S(X, h)} \#\Sel{\varphi_1}{E_n}
\end{align*}
by counting the trivial element of $\Sel{\varphi_1}{E_n}$.
Hence, 
we complete the proof.

In a similar manner,
we also obtain the same arguments when we consider the other cases $h = \pm 2$ and $i=2,3$.
\end{proof}

Now, we will prove \cref{main1}.
We can rewrite \cref{main1} explicitly as
\begin{align}\label{main1explicit}
\lim_{X\to\infty} \frac{\#\{f\in \mathcal{F}^{0,\text{sol}}_{4n^2}\mid \text{$n$: sqf.\ and $H_{\text{BH}}(f)=4n^2<X$}\}}{\#\{f\in \mathcal{F}^{0,\text{ls}}_{4n^2}\mid \text{$n$: sqf.\ and $H_{\text{BH}}(f)=4n^2<X$}\}} = 1.
\end{align}
The left hand side quantity is equal to
\begin{align*}
\lim_{X\to\infty}\dfrac{\displaystyle \sum_{\substack{0<n<\sqrt{X}/2\\ \text{$n$: sqf.}}}\#\mathcal{F}^{0, \text{sol}}_{4n^2}}{\displaystyle \sum_{\substack{0<n<\sqrt{X}/2\\ \text{$n$: sqf.}}}\#\mathcal{F}^{0, \text{ls}}_{4n^2}},
\end{align*}
so that by writing $X$ instead of $\sqrt{X}/2$, 
we only have to show the following theorem.
\begin{theorem}\label{main1refined}
\begin{align*}
\lim_{X\to\infty}\dfrac{\displaystyle \sum_{\substack{0<n<X\\ \text{$n$: sqf.}}}\#\mathcal{F}^{0, \text{sol}}_{4n^2}}{\displaystyle \sum_{\substack{0<n<X\\ \text{$n$: sqf.}}}\#\mathcal{F}^{0, \text{ls}}_{4n^2}}=1
\end{align*}
\end{theorem}

\begin{proof}
Using \cref{integralSummarized}, we can freely interpret elements of Selmer group $\Sel{\varphi_1}{E_n}$ as elements in $\mathcal{F}^{0, \text{ls}}_{4n^2}$ under the bijection in \cref{correspond}.

In \cref{Selmergrp},
the main term $\#S(X,h)$ comes from identity elements of Selmer groups $\Sel{\varphi_1}{E_n}$, and the other elements are negligible.
Since elements of $\mathcal{F}^{0, \text{ls}}_{4n^2}$ which come from identities are soluble,
the inequalities
\[
    \begin{aligned}
    \sum_{h = \pm 1, \pm 2, \pm 3} \#S(X, h) &= \sum_{\substack{0 < n < X \\ \text{$n$: sqf.}}} 1 \\
    &\le 
    \sum_{\substack{0 < n < X \\ \text{$n$: sqf.}}} \#\mathcal{F}^{0, \text{sol}}_{4n^2}\\
    &\le
    \sum_{\substack{0 < n < X \\
    \text{$n$: sqf.}}}\#\mathcal{F}^{0, \text{ls}}_{4n^2}\\
    &\le \sum_{h = \pm 1, \pm 2, \pm 3} \#S(X, h) + o(X)
    \end{aligned}
\] 
hold as $X\to\infty$.
Hence,
we obtain
\begin{align*}
    \frac{\displaystyle \sum_{\substack{0<n<X\\ \text{$n$: sqf.}}}\#\mathcal{F}^{0, \text{sol}}_{4n^2}}
    {\displaystyle \sum_{\substack{0<n<X\\ \text{$n$: sqf.}}}\#\mathcal{F}^{0, \text{ls}}_{4n^2}} \to 1 
\end{align*}
as $X \to \infty$.
\end{proof}

In a similar manner, 
we also obtain the following theorem for $\Sel{\varphi_2}{E_n}$ and $\Sel{\varphi_3}{E_n}$.

\begin{theorem}\label{main1variant}
\begin{align*}
\lim_{X\to\infty}\frac{\displaystyle\sum_{\substack{0<n<X\\ \text{$n$: sqf.}}}\#\mathcal{F}^{6n, \text{sol}}_{n^2}}{\displaystyle\sum_{\substack{0<n<X\\ \text{$n$: sqf.}}}\#\mathcal{F}^{6n, \text{ls}}_{n^2}}=\lim_{X\to\infty}\frac{\displaystyle\sum_{\substack{0<n<X\\ \text{$n$: sqf.}}}\#\mathcal{F}^{-6n, \text{sol}}_{n^2}}{\displaystyle\sum_{\substack{0<n<X\\ \text{$n$: sqf.}}}\#\mathcal{F}^{-6n, \text{ls}}_{n^2}}=1.
\end{align*}
\end{theorem}

%----------------------
\section{Proof of \cref{main2}}\label{proofofmain2}
%----------------------
In Subsection \ref{2Coverings}, 
we employ results on $2$-Selmer groups to estimate soluble binary quartic forms.
Later in Subsection \ref{proofmain2},
we prove \cref{main2}.

%----------------------
\subsection{Results on 2-coverings}\label{2Coverings}
%----------------------

We set $n=D_1D_2D_3D_4$ for pairwise coprime integers $D_1, D_2, D_3, D_4 > 0$ and $\mathcal{C} = \mathcal{C}(D_1, D_2, D_3, D_4)$ is the genus one curve defined by
\begin{align*}
\mathcal{C}(D_1, D_2, D_3, D_4) \colon
\begin{cases}
    D_1X^2+D_4W^2=D_2Y^2\\
    D_1X^2-D_4W^2=D_3Z^2.
\end{cases}
\end{align*}

As the following lemma says, this curve represents $2$-coverings of $E$.

\begin{lemma}(cf.\ \cite[X.4.5.1]{SilvermanAEC}, \cite[Lemma 1]{Heath-Brown1993})\label{Sel2ToDiagonals}
Let $n$ be a squarefree integer.
Then, the $2$-Selmer group $\Sel{2}{E_n}$ is bijective to the set
\begin{align*}
    \left\{(D_1, D_2, D_3, D_4) \setmid
    \begin{gathered}
    \text{$D_1, D_2, D_3, D_4\in \Z_{>0}$ are pairwise coprime,}\\
    n=D_1D_2D_3D_4,\\
    \text{$\mathcal{C}(D_1, D_2, D_3, D_4)(\Q_v) \neq \emptyset$ for all places $v$ of $\Q$}
    \end{gathered}
    \right\}.
\end{align*}
\end{lemma}

We quote the following lemma about the existence of local solutions for $\mathcal{C}$.
Although Heath-Brown treated only the case when $n$ is odd in \cite{Heath-Brown1993},
the following statements are valid when $n$ is even and squarefree.

\begin{lemma}(\cite[pp 175-176]{Heath-Brown1993})\label{Sel2localcondition}
Let $p$ be $\infty$ or a odd prime.
\begin{enumerate}
    \item $\mathcal{C}(\Q_{\infty})\neq \emptyset$ where $\Q_{\infty}=\R$.
    \item If $p\nmid 2n$, then $\mathcal{C}(\Q_p)\neq \emptyset$.
    \item If $p\mid D_1$, $\mathcal{C}(\Q_p)\neq\emptyset \iff \left(\frac{D_4D_2}{p}\right)=\left(\frac{-D_4D_3}{p}\right)=1$.
    \item If $p\mid D_2$, $\mathcal{C}(\Q_p)\neq\emptyset \iff \left(\frac{-D_1D_4}{p}\right)=\left(\frac{2D_1D_3}{p}\right)=1$.
    \item If $p\mid D_3$, $\mathcal{C}(\Q_p)\neq\emptyset \iff \left(\frac{D_1D_4}{p}\right)=\left(\frac{2D_1D_2}{p}\right)=1$.
    \item If $p\mid D_4$, $\mathcal{C}(\Q_p)\neq\emptyset \iff \left(\frac{D_1D_2}{p}\right)=\left(\frac{D_1D_3}{p}\right)=1$.
\end{enumerate}
\end{lemma} 

\begin{lemma}\label{Heath-Brown}
For $h = \pm 1, \pm 2, \pm 3$, we have
\begin{align*}
    \sum_{n\in S(X,h)}\#\Sel{2}{E_{n}} \le 12\#S(X,h) + o(X)
\end{align*}
as $X\to\infty$.
\end{lemma}

\begin{proof}
For the case $h$ is odd, more strict estimate is given in \cite[Theorem 1]{Heath-Brown1993}.
To consider the case $h$ is even, we briefly sketch the proof; for details, see \cite{Heath-Brown1993}.

By \cref{Sel2ToDiagonals}, it is enough to count the quadruples of pairwise coprime positive integers $(D_1, D_2, D_3, D_4)$ such that $\mathcal{C}(D_1, D_2, D_3, D_4)$ has $\Q_v$-rational points for any place $v$ of $\Q$.

For odd primes $v$ and $v = \infty$, \cref{Sel2localcondition} gives the condition so that the curve $\mathcal{C}(D_1, D_2, D_3, D_4)$ has a $\Q_v$-rational point. 
For $v=2$ case, \cite[Lemma 2]{Heath-Brown1993} states that the existence of $\Q_2$-rational points in $\mathcal{C}(D_1, D_2, D_3, D_4)$ is determined for the existence of $\Q_v$-rational points for other $v$ (including $v = \infty$).
By counting coefficients satisfying the local conditions for any odd prime $v$ and $v = \infty$, we obtain the desired estimation.

For the case $h$ is even, 
again we apply \cref{Sel2ToDiagonals} and reduce to count the quadruples of integers. 
For odd primes $v$ or $v =\infty$, the conditions for the existence of $\Q_v$-rational points are the same as \cref{Sel2localcondition}. 
To obtain an estimation from above, we ignore the condition on $v=2$.
\end{proof}

%----------------------
\subsection{Proof of \cref{main2}}\label{proofmain2}
%----------------------

First, we show the following lemma.
This lemma is essential when we evaluate the number of locally soluble integral quartics.
\begin{lemma}[{\cite[p.\ 47.\ (9)]{Xiong-Zaharescu2009}}]\label{localconditiongeneral}
Define 
\[
    s(n, \widehat{\varphi_i}) \coloneqq
    \dim_{\F_2} \Sel{\widehat{\varphi_i}}{E_{i,n}} - 2.
\]
For each $h = \pm 1, \pm 3$ and $i\in \{1,2,3\}$, 
we obtain
\begin{align*}
\sum_{n\in S(X, h)} s(n, \widehat{\varphi_i}) = \frac{\#S(X, h)\log\log X}{2} + O(X)
\end{align*}
as $X\to \infty$. 
\end{lemma}

Before we begin to prove \cref{main2}, 
we show the explicit statement of \cref{main2}.

As a similar manner to (\ref{main1explicit}),
we can rewrite \cref{main2} explicitly as
\begin{align*}
\lim_{X\to\infty} \frac{\#\{f\in \mathcal{F}^{0,\text{sol}}_{-n^2}\mid \text{$n$: sqf.\ and $H_{\text{BH}}(f)=n^2<X$}\}}{\#\{f\in \mathcal{F}^{0,\text{ls}}_{-n^2}\mid \text{$n$: sqf.\ and $H_{\text{BH}}(f)=n^2<X$}\}} = 0.
\end{align*}
Hence, by changing $\sqrt{X}$ to $X$,
it is sufficient to show the following theorem.
\begin{theorem}\label{phihat1}
\begin{align*}
\lim_{X\to\infty}\frac{\displaystyle\sum_{\substack{0<n<X\\ \text{$n$: sqf.}}}\#\mathcal{F}^{0, \text{sol}}_{-n^2}}{\displaystyle\sum_{\substack{0<n<X\\ \text{$n$: sqf.}}}\#\mathcal{F}^{0, \text{ls}}_{-n^2}}=0
\end{align*}
\end{theorem}

\begin{proof}
We estimate the denominator and numerator independently.

First, consider the denominator.
From \cref{correspond,localconditiongeneral},
we obtain
\begin{align*}
    \sum_{\substack{0<n<X\\ \text{$n$: sqf.}}}\#\mathcal{F}^{0, \text{ls}}_{-n^2}
    &\geq \sum_{\substack{0<n<X\\ \text{$n$: sqf.\ and odd}}} \#\Sel{\widehat{\varphi_1}}{E_{1, n}}\\
    &= \sum_{\substack{0<n<X\\ \text{$n$: sqf.\ and odd}}}
    2^{2+s(n, \widehat{\varphi_1})}\\
    &= 4\sum_{\substack{0<n<X\\ \text{$n$: sqf.\ and odd}}}
    (1+1)^{s(n, \widehat{\varphi_1})}\\
    &\geq 4\sum_{\substack{0<n<X\\ \text{$n$: sqf.\ and odd}}} (1 + s(n, \widehat{\varphi_1}))\\
    &=4 \times 4 \times \left(1 + \frac{X\log\log X}{2\pi^2}\right) + o(X)\\
    &= \frac{8X\log\log X}{\pi^2} + o(X)
\end{align*}
as $X\to \infty$.

Next, we consider the numerator.
By \cref{integralSummarized}, the natural map
\begin{align*}
    \mathcal{F}^{0, \text{sol}}_{-n^2}\to 
    W^{0}_{-n^2}(\Q)^{\text{sol}}/{\sim}
\end{align*}
is bijective for each $n$.
By \cref{correspond}, the set $W^{0}_{-n^2}(\Q)^{\text{sol}}/{\sim}$ is bijective to the weak Mordell--Weil group $E_n(\Q) / \widehat{\varphi_1}(E_{1,n}(\Q))$.
Since there are a surjection
\begin{align*}
    E_{n}(\Q)/\widehat{\varphi_1}(E_{1,n}(\Q)) 
    &\twoheadleftarrow E_n(\Q)/2E_n(\Q)
\end{align*}
and an inclusion
\begin{align*}
    E_n(\Q)/2E_n(\Q)
    &\subseteq \Sel{2}{E_n},
\end{align*}
we can bound $\# \mathcal{F}^{0, \text{sol}}_{-n^2}$ for each $n$ from above as
\begin{align*}
    \#\mathcal{F}^{0, \text{sol}}_{-n^2}
    \le \# \Sel{2}{E_n}.
\end{align*}
Summing up among squarefree $n$, we obtain 
\begin{align*}
    \sum_{\substack{0<n<X\\ \text{$n$: sqf.}}}\#\mathcal{F}^{0, \text{sol}}_{-n^2}\leq \sum_{\substack{0<n<X\\ \text{$n$: sqf.}}} \#\Sel{2}{E_n}\leq O(X)
\end{align*}
as $X\to\infty$ by \cref{Heath-Brown}.
Hence, we conclude that 
\begin{align*}
\frac{\displaystyle \sum_{\substack{0<n<X\\ \text{$n$: sqf.}}} \#\mathcal{F}^{0, \text{sol}}_{-n^2}}{\displaystyle \sum_{\substack{0<n<X\\ \text{$n$: sqf.}}} \#\mathcal{F}^{0, \text{ls}}_{-n^2}}\leq \frac{O(X)}{8X\log\log X/ \pi^2 + O(X)} \to 0
\end{align*}
as $X\to\infty$ and complete the proof.
\end{proof}

In a similar manner to the case $i=1$,
we can also show the case $i=2$ and $3$.

\begin{theorem}\label{phihat2and3}
\begin{align*}
\lim_{X\to\infty}\frac{\displaystyle\sum_{\substack{0<n<X\\ \text{$n$: sqf.}}}\#\mathcal{F}^{3n, \text{sol}}_{2n^2}}{\displaystyle\sum_{\substack{0<n<X\\ \text{$n$: sqf.}}}\#\mathcal{F}^{3n, \text{ls}}_{2n^2}}=
\lim_{X\to\infty}\frac{\displaystyle\sum_{\substack{0<n<X\\ \text{$n$: sqf.}}}\#\mathcal{F}^{-3n, \text{sol}}_{2n^2}}{\displaystyle\sum_{\substack{0<n<X\\ \text{$n$: sqf.}}}\#\mathcal{F}^{-3n, \text{ls}}_{2n^2}}=0
\end{align*}
\end{theorem}

%----------------------
\section{Proof of \cref{main3}}\label{proofofmain3}
%----------------------

In the following,
we only consider the case $i=1$.
There exists a natural homomorphism
\begin{align*}
\pi_1 \colon \Sel{2}{E_n}\to\Sel{\widehat{\varphi_1}}{E_{1,n}}. 
\end{align*}
We say a binary quartic form $f\in W^{0}_{-n^2}(\Z)^{\text{ls}}$ is \emph{strictly locally soluble} if the element in Selmer group $\Sel{\widehat{\varphi_1}}{E_{1,n}}$ corresponding to $f$ is in the image of $\pi_1$.
We write the subset of $W^{0}_{-n^2}(\Z)$ of strictly locally soluble binary quartic forms as $W^{0}_{-n^2}(\Z)^{\text{sls}}$.

The set $W^{0}_{-n^2}(\Z)^{\text{sol}}$ is a subset of $W^{0}_{-n^2}(\Z)^{\text{sls}}$.
In fact, this is followed from \cref{correspond} and the following commutative diagram (cf.\ \cite[p.\ 97]{Aoki}):
\begin{center}
\begin{tikzpicture}[auto]
\node (11) at (-0.7, 1.5) {$0$}; 
\node (12) at (1.5, 1.5) {$E_{1,n}(\Q)/\varphi_1(E_n(\Q))$}; 
\node (13) at (5, 1.5) {$E_{n}(\Q)/2E_n(\Q)$};
\node (14) at (8.5, 1.5) {$E_{n}(\Q)/\widehat{\varphi_1}(E_{1,n}(\Q))$};
\node (21) at (-0.7, 0) {$0$};
\node (22) at (1.5, 0) {$\Sel{\varphi_1}{E_n}$}; 
\node (23) at (5, 0) {$\Sel{2}{E_n}$}; 
\node (24) at (8.5, 0) {$\Sel{\widehat{\varphi_1}}{E_{1,n}}$.}; 
\node (15) at (10.7, 1.5) {$0$};

\draw[->] (11) to (12); \draw[->] (12) to (13); \draw[->] (13) to (14);
\draw[->] (21) to (22); \draw[->] (22) to (23); \draw[->] (23) to[edge label={\small $\pi_1$}] (24);
\draw[->] (14) to (15); \draw[->] (12) to (22); \draw[->] (13) to (23);
\draw[->] (14) to (24);
\end{tikzpicture}
\end{center}

In a similar manner to $\mathcal{F}^{0, \text{sol}}_{-n^2}$ or $\mathcal{F}^{0, \text{ls}}_{-n^2}$, 
we define $\mathcal{F}^{0, \text{sls}}_{-n^2}$ as $\mathcal{F}^{0}_{-n^2}\cap W^{0}_{-n^2}(\Z)^{\text{sls}}$.
Now we restate \cref{main3} in a rigorous form.
As in \cref{main1,main2},
it is sufficient to show the following theorem.
\begin{theorem}\label{keytheorem}
We have
\begin{align*}
\liminf_{X\to\infty}\frac{\displaystyle\sum_{\substack{0<n<X\\ \text{$n$: sqf.}}}\#\mathcal{F}^{0, \text{sol}}_{-n^2}}{\displaystyle\sum_{\substack{0<n<X\\ \text{$n$: sqf.}}}\#\mathcal{F}^{0, \text{sls}}_{-n^2}}
\geq \frac{2.559}{6}.
\end{align*}
\end{theorem}

In the rest of this section, 
we will prove this theorem.
To prove it, we estimate the denominator and numerator independently.
First, consider the denominator.

The number of strictly locally soluble forms are bounded above by the number of elements of $2$-Selmer groups.
This leads us to estimate
\begin{align}\label{denominator}
\sum_{\substack{0<n<X\\ \text{$n$: sqf.}}}\#\mathcal{F}^{0, \text{sls}}_{-n^2}\notag
&\leq \sum_{h\in\{\pm 1, \pm 2, \pm 3\}}\sum_{n\in S(X, h)}\#\Sel{2}{E_n}\\
&\leq 6\times 12\times \frac{1}{\pi^2}X + o(X)
\end{align}
as $X\to\infty$. 
Here, the last inequality follows from \cref{Selmergrp}.

Next, consider the numerator.
When we evaluate the number of soluble forms,
the following proposition plays an important role. 

\begin{proposition}[{\cite[Theorem 1.5]{Smith}, cf.\ \cite[Theorem 5.3]{Li2020}}]\label{goldfeldconj}
Consider the family of elliptic curves $\{E_n \mid \text{$n$: sqf. and $n\equiv -1, -2, -3 \pmod{8}$}\}$. 
In this family, the proportion of the curves with $\rank_{\Q} (E_n) = 1$ is at least $55.9 \%$.
\end{proposition}

\begin{proof}
By \cite[Theorem 1.5]{Smith},
the analytic rank of $E_n$ is equal to one for at least $62.9\%, 41.9\%$ and $62.9\%$ of squarefree integers $n$ with $n\equiv -1, -2$ and $-3 \pmod{8}$ respectively.
Combining with the results of Gross--Zagier \cite{Gross-Zagier} and Kolyvagin \cite{Kolyvagin}, we deduce that the same is true for the arithmetic rank.
\end{proof}

We need another proposition on the structure of the Mordell--Weil groups of $E_n$.
Let us write the torsion part of $E_n(\Q)$ as $T$, and the torsion part of $E_{1,n}(\Q)$ as $T_1$. 
Since $\widehat{\varphi_1}$ is a morphism of degree $2$, we have $\widehat{\varphi_1}(T_1) \subseteq T$.
Moreover, we can take subgroups $A \subseteq E_n(\Q)$ and $A_1 \subseteq E_{1,n}(\Q)$ such that
\begin{align*}
    E_{1,n}(\Q) \cong T_1 \times A_1, \quad
    E_{n}(\Q) \cong T \times A
\end{align*}
and $\widehat{\varphi_1}(A_1) \subseteq A$. 
Note that $A, A_1$ are free of rank $r$.
By these, we can consider the decomposition
\begin{align*}
    E_n(\Q)/\widehat{\varphi_1}(E_{1,n}(\Q)) &\cong 
    T/\widehat{\varphi_1}(T_1) \times A/\widehat{\varphi_1}(A_1), \\
    E_n(\Q)/2E_n(\Q) &\cong 
    T/2T \times A/2A.
\end{align*}
The following proposition states that
the contribution of the torsion parts does not depend on $n > 1$. 

\begin{proposition}\label{MordellWeil}
When $n > 1$, we have
\begin{align*}
    \dim_{\F_2} A/\widehat{\varphi_1}(A_1) &= \dim_{\F_2} \frac{E_n(\Q)}{\widehat{\varphi_1}(E_{1,n}(\Q))} - 2, \\
    r = \dim_{\F_2} A/2A &= \dim_{\F_2} \frac{E_n(\Q)}{2E_n(\Q)} - 2.
\end{align*}
\end{proposition}

\begin{proof}
The natural inclusion $E_n(\Q)[2] \hookrightarrow E_n(\Q)$ and the natural surjection $E_n(\Q) \twoheadrightarrow E_n(\Q)/2E_n(\Q)$ give a composite map
\begin{align}\label{CompositeMap}
    E_n(\Q)[2] \to
    \frac{E_n(\Q)}{2E_n(\Q)}.
\end{align}
Since $\widehat{\varphi_1} \circ \varphi_1 = [2]$, there is also a natural surjection
\begin{align}\label{NaturalSurj}
    \frac{E_n(\Q)}{2E_n(\Q)} \twoheadrightarrow
    \frac{E_n(\Q)}{\widehat{\varphi_1}(E_{1,n}(\Q))}.
\end{align}
By the above decomposition, the maps \eqref{CompositeMap} and \eqref{NaturalSurj} give two homomorphisms
\begin{align}
    E_n(\Q)[2] &\to T/2T, \label{CompositeMap2} \\
    T/2T &\twoheadrightarrow
    T/\widehat{\varphi_1}(T_1). \label{NaturalSurj2}
\end{align}
Recall that $\#E_n(\Q)[2] = 4$.
Thus, in order to prove the proposition, 
it is sufficient to show that $E_n(\Q)[2]$, $T/2T$ and $T/\widehat{\varphi_1}(T_1)$ are isomorphic by the homomorphisms \eqref{CompositeMap2} and \eqref{NaturalSurj2}.

We show that the composition map of \eqref{CompositeMap2} and \eqref{NaturalSurj2} is injective.
This amounts to show $\widehat{\varphi_1}(E_{1,n}(\Q)) \cap E_n(\Q)[2] = \{\infty\}$.
Since the degree of $\widehat{\varphi_1}$ is two, we have
\begin{align*}
    \widehat{\varphi_1}(E_{1,n}(\Q)) \cap E_n(\Q)[2] \subseteq
    \widehat{\varphi_1}(E_{1,n}(\Q)[4]) \cap E_n(\Q)[2].
\end{align*}
Since $\Ker(\widehat{\varphi_1}) = E_{1,n}(\Q)[2] = E_{1,n}(\Q)[4]$, we have
\begin{align*}
    \widehat{\varphi_1}(E_{1,n}(\Q)[4]) \cap E_n(\Q)[2] 
    &= \widehat{\varphi_1}(E_{1,n}(\Q)[2]) \cap E_n(\Q)[2] \\
    &= \{\infty\} \cap E_n(\Q)[2] = \{\infty\}.
\end{align*}
Hence we obtain $\widehat{\varphi_1}(E_{1,n}(\Q)) \cap E_n(\Q)[2] = \{\infty\}$.

Next, we show that the map \eqref{CompositeMap2} is an isomorphism.
Since the compositon map of \eqref{CompositeMap2} and \eqref{NaturalSurj2} is injective,
the map \eqref{CompositeMap2} is also injective.
The groups $E_n(\Q)[2] = T[2]$ and $T/2T$ have the same order since they are the kernel and the cokernel of an endomorphism $T \overset{2}{\to} T$ of a finite abelian group $T$.
Hence the map \eqref{CompositeMap2} is surjective.

Combining these with the fact that the map \eqref{NaturalSurj2} is surjective,
we obtain the maps \eqref{CompositeMap2} and \eqref{NaturalSurj2} are isomorphisms.
This completes the proof.
\end{proof}

\begin{remark}
Under the bijection in \cref{correspond},
all $2$-torsion points
\[
\infty, (0,0), (n,0), (-n,0) \in E_n(\Q)[2]
\]
correspond to
\[
x^4-n^2y^4, -x^4+n^2y^4, nx^4-ny^4, -nx^4+ny^4\in \mathcal{F}^{0,\text{sol}}_{-n^2}.
\]
For details,
see \cite[Proposition X.4.9]{SilvermanAEC}.
\end{remark}

Using the above propositions,
we can show the following proposition which estimates the numerator.
\begin{proposition}\label{numeratorkai}
We have 
\begin{align*}
\sum_{\substack{0<n<X\\ \text{$n$: sqf.}}}\#\mathcal{F}^{0, \text{sol}}_{-n^2} \geq \frac{4(6 + 0.559\times 3)}{\pi^2}X + o(X)
\end{align*}
as $X\to\infty$.
\end{proposition}

\begin{proof}
First, 
consider the exact sequence (\cite[Remark X.4.7]{SilvermanAEC})
\[
0\to \frac{E_{1,n}(\Q)[\widehat{\varphi_1}]}{\varphi_1(E_n(\Q)[2])} \to \frac{E_{1,n}(\Q)}{\varphi_1(E_n(\Q))} \to \frac{E_n(\Q)}{2E_n(\Q)} \to \frac{E_n(\Q)}{\widehat{\varphi_1}(E_{1,n}(\Q))} \to 0.
\]
For $n > 1$, we have
\[
    \frac{E_{1,n}(\Q)[\widehat{\varphi_1}]}{\varphi_1(E_n(\Q)[2])} = 0.
\]
By the above exact sequence, if $n > 1$, we obtain 
\begin{align*}%\label{MWequation}
    \dim_{\F_2}\frac{E_{1,n}(\Q)}{\varphi_1(E_n(\Q))} + \dim_{\F_2}\frac{E_n(\Q)}{\widehat{\varphi_1}(E_{1,n}(\Q))} = \dim_{\F_2}\frac{E_n(\Q)}{2E_n(\Q)}.
\end{align*} 
\item Recall that $E_{1,n}(\Q)/\varphi_1(E_n(\Q)) \subseteq \Sel{\varphi_1}{E_n}$. As stated in \cref{Selmergrp}, we have 
\[
    \sum_{\substack{0 < n < X \\ \text{$n$: sqf.}}} \# (\Sel{\varphi_1}{E_n} \setminus \{0\}) = o(X).
\]
as $X\to\infty$.
Hence we obtain 
\begin{align*}
    \sum_{\substack{0 < n < X \\ \text{$n$: sqf.}}} \dim_{\F_2}\frac{E_{1,n}(\Q)}{\varphi_1(E_n(\Q))} = o(X)
\end{align*}
and
\begin{align}\label{ObtainedEstimate}
    \sum_{\substack{0 < n < X \\ \text{$n$: sqf.}}}
    \dim_{\F_2} \frac{E_n(\Q)}{\widehat{\varphi_1}(E_{1,n}(\Q))} 
    &= \sum_{\substack{0 < n < X \\ \text{$n$: sqf.}}} \dim_{\F_2}\frac{E_n(\Q)}{2E_n(\Q)} + o(X).
\end{align}
as $X \to \infty$.

Since
\begin{align*}
    \#\frac{E_n(\Q)}{\widehat{\varphi_1}(E_n(\Q))}
    &= \# (E_n(\Q)[2] \times A/\widehat{\varphi_1}(A_1)) \\
    &= \# E_n(\Q)[2] \times 2^{\dim_{\F_2} A/\widehat{\varphi}(A_1)} \\
    &\ge 4 \times (1 + {\dim_{\F_2} A/\widehat{\varphi}(A_1)}),
\end{align*}
and combining with \cref{correspond,integralSummarized}, 
the inequalities
\begin{align*}
    \sum_{\substack{0<n<X\\ \text{$n$: sqf.}}}\#\mathcal{F}^{0, \text{sol}}_{-n^2}
    &\geq \sum_{\substack{0 < n < X \\ \text{$n$: sqf.}}}
    \#\frac{E_n(\Q)}{\widehat{\varphi_1}(E_{1,n}(\Q))}\\
    &\geq \sum_{\substack{0 < n < X \\ \text{$n$: sqf.}}} 4 \times (1 + {\dim_{\F_2} A/\widehat{\varphi_1}(A_1)})
\end{align*}
hold.
Moreover, the inequalities
\begin{align*}
    \sum_{\substack{0 < n < X \\ \text{$n$: sqf.}}} 4 \times (1 + {\dim_{\F_2} A/\widehat{\varphi}(A_1)})
    &\geq \sum_{\substack{0 < n < X \\ \text{$n$: sqf.}}} 4 \times (1 + r) + o(X) \\
    &\geq \Bigg(\sum_{\substack{0<n<X\\ \text{$n$: sqf.}}} + \sum_{\substack{0<n<X\\ \text{$n$: sqf.}\\ n\equiv 5,6,7 \bmod{8}\\ \rank E_{n}(\Q) = 1}} \Bigg)  4 + o(X)\\
    &\geq \frac{4(6 + 0.559\times 3)}{\pi^2}X + o(X)
\end{align*}
also hold as $X\to \infty$.
Here, the first inequality follows from \eqref{ObtainedEstimate} and \cref{MordellWeil} and the second inequality follows from \cref{goldfeldconj}.
Therefore, 
we obtain the inequality 
\[
\sum_{\substack{0<n<X\\ \text{$n$: sqf.}}}\#\mathcal{F}^{0, \text{sol}}_{-n^2} \geq \frac{4(6 + 0.559\times 3)}{\pi^2}X + o(X)
\]
as $X\to\infty$ and we complete the proof.
\end{proof}

Combining with (\ref{denominator}) and \cref{numeratorkai}, 
we obtain 
\begin{align*}
\liminf_{X\to\infty}\frac{\displaystyle\sum_{\substack{0<n<X\\ \text{$n$: sqf.}}}\#\mathcal{F}^{0, \text{sol}}_{-n^2}}{\displaystyle\sum_{\substack{0<n<X\\ \text{$n$: sqf.}}}\#\mathcal{F}^{0, \text{sls}}_{-n^2}}
\geq \frac{4(6+0.559\times 3)/\pi^2}{6\times 12 \times (1/\pi^2)}=\frac{2.559}{6}\sim 0.4265
\end{align*}
and we completed the proof of \cref{keytheorem}.

In a similar manner,
we define strictly locally soluble quartics in $W^{\pm 3n}_{2n^2}(\Z)^{\text{ls}}$ and consider the subsets $W^{\pm 3n}_{2n^2}(\Z)^{\text{sls}} \subseteq W^{\pm 3n}_{2n^2}(\Z)^{\text{ls}}$ of all strictly locally soluble quartics.
We also define $\mathcal{F}^{3n, \text{sls}}_{2n^2} =\mathcal{F}^{3n}_{2n^2}\cap W^{3n}_{2n^2}(\Z)^{\text{sls}}$ and $\mathcal{F}^{-3n, \text{sls}}_{2n^2} =\mathcal{F}^{-3n}_{2n^2}\cap W^{-3n}_{2n^2}(\Z)^{\text{sls}}$. 
Then, 
the following theorem also holds.

\begin{theorem}
We have
\begin{align*}
\liminf_{X\to\infty}\frac{\displaystyle\sum_{\substack{0<n<X\\ \text{$n$: sqf.}}}\#\mathcal{F}^{3n, \text{sol}}_{2n^2}}{\displaystyle\sum_{\substack{0<n<X\\ \text{$n$: sqf.}}}\#\mathcal{F}^{3n, \text{sls}}_{2n^2}} \geq \frac{4(6+0.559\times 3)/\pi^2}{6\times 12 \times (1/\pi^2)}=\frac{2.559}{6}\sim 0.4265
\end{align*}
and
\begin{align*}
\liminf_{X\to\infty}\frac{\displaystyle\sum_{\substack{0<n<X\\ \text{$n$: sqf.}}}\#\mathcal{F}^{-3n, \text{sol}}_{2n^2}}{\displaystyle\sum_{\substack{0<n<X\\ \text{$n$: sqf.}}}\#\mathcal{F}^{-3n, \text{sls}}_{2n^2}} \geq \frac{4(6+0.559\times 3)/\pi^2}{6\times 12 \times (1/\pi^2)}=\frac{2.559}{6}\sim 0.4265
\end{align*}
\end{theorem}
The proof is the same as \cref{keytheorem}.

\begin{remark}
Assume that the Goldfeld conjecture holds.
Then, the proportion of the elliptic curves $\{E_n \mid \text{$n$: sqf.\ and $n\equiv -1, -2, -3\pmod{8}$}\}$ which satisfy $\rank E_n(\Q) = 1$ is 1.
Hence, we obtain 
\begin{align*}
    \sum_{\substack{0<n<X\\ \text{$n$: sqf.}}}\#\mathcal{F}^{0, \text{sol}}_{-n^2}\geq \Bigg(\sum_{\substack{0<n<X\\ \text{$n$:  sqf.}}} + \sum_{\substack{0<n<X\\ \text{$n$: sqf.}\\ n\equiv 5,6,7 \bmod{8}\\ \rank E_{n}(\Q) = 1}} \Bigg)  4 \geq \frac{4(6 + 3)}{\pi^2}X + o(X)
\end{align*}
as $X\to \infty$ and 
\begin{align*}
\liminf_{X\to\infty}\frac{\displaystyle\sum_{\substack{0<n<X\\ \text{$n$: sqf.}}}\#\mathcal{F}^{0, \text{sol}}_{-n^2}}{\displaystyle\sum_{\substack{0<n<X\\ \text{$n$: sqf.}}}\#\mathcal{F}^{0, \text{sls}}_{-n^2}}
\geq \frac{4(6+3)/\pi^2}{6\times 12 \times (1/\pi^2)}=\frac{1}{2}.
\end{align*}
We expect the limit exists, and coincides with $1/2$.
\end{remark}

%----------------------
\section*{Acknowledgments}
%----------------------
The authors would like to thank Prof.\ Ken-ichi Bannai for his careful reading and valuable comments on a draft of this paper.
The first author was supported by JSPS KAKENHI Grant Number 21K13773.
The second author was supported by JSPS KAKENHI Grant Number 22J13607.
%----------------------
%References
%----------------------

\end{document}